\documentclass{amsart}

\usepackage{amssymb}
\usepackage{latexsym}
\usepackage{amsmath}
\usepackage{enumerate}
\usepackage{amsthm}

\newcommand{\shapipe}{\,
                    \setlength{\unitlength}{1ex}
                    \begin{picture}(2,2)
                    \put(.25,.15){\line(0,1){1.22}}
                    \put(.25,1.37){\line(1,0){0.61}}
                    \put(.45,-.05){\line(1,0){0.61}}
                    \put(.4,0){\line(1,0){0.61}}
                    \put(.35,.05){\line(1,0){0.61}}
                    \put(.3,.1){\line(1,0){0.61}}
                    \put(.25,.15){\line(1,0){0.61}}
                    \put(1.04,-.05){\line(0,1){1.22}}
                    \put(0.99,0){\line(0,1){1.22}}
                    \put(0.94,.05){\line(0,1){1.22}}
                    \put(0.89,.1){\line(0,1){1.22}}
                    \put(0.84,.15){\line(0,1){1.22}}
                    \end{picture}
                    \!}

\newcommand{\qee} {\hspace*{2mm}\hfill $\shapipe$}

\newtheorem{theorem}{Theorem}[section]
\newtheorem{ques}[theorem]{Open Question}
\newenvironment{question}{\begin{ques} \rm}{\qee\end{ques}}
\newtheorem{lem}[theorem]{Lemma}

\newtheorem{cor}[theorem]{Corollary}

\newtheorem{factief}[theorem]{Fact}

\newtheorem{rem}[theorem]{Remark}
\newenvironment{remark}{\begin{rem} \rm}{\qee\end{rem}}

  \newcommand{\bident}[0]{\bigskip\noindent}
\newcommand{\medent}{\medskip\noindent}
\newcommand{\verz}[1]{\{ #1 \}}
\renewcommand{\iff}{\leftrightarrow}
\newcommand{\Iff}{\Leftrightarrow}
\newcommand{\To}{\Rightarrow}

\newcommand{\opr}{\mbox{\small $\Box$}}
\newcommand{\oco}{\Diamond}

 \newcommand{\gn}[1]{{\ulcorner #1 \urcorner}}
 \newcommand\textvtt[1]{{\normalfont\fontfamily{cmvtt}\selectfont #1}}
   \newcommand{\seli}[1]{\textvtt{\textup{[}\,#1\,\textup{]}}}
  \newcommand{\gtwo}{Second Incompleteness Theorem}

\title[From Tarski to G\"odel]{From Tarski to G\"odel \\
\hspace{1cm}\\
{\footnotesize Or, how to derive the Second Incompleteness Theorem from the
Undefinability of Truth without Self-reference}}

\author{Albert Visser}
 \address{Philosophy, Faculty of Humanities,
                Utrecht University,
               Janskerkhof 13,
                3512BL~~Utrecht, The Netherlands}
\email{a.visser@uu.nl}
\date{\today}



\thanks{I am grateful to Fedor Pakhomov for his comments on an earlier version of the paper and for his
encouragement,}

\begin{document}

\begin{abstract}
In this paper, we provide a fairly general self-reference-free proof of the Second Incompleteness Theorem
from Tarski's Theorem of the Undefinability of Truth.
\end{abstract}

\maketitle

\section{Prelude}
Self-reference is a great and beautiful thing, but it may be interesting to see what one can do  without it.
In this paper,  we provide a self-reference-free proof of the Second Incompleteness Theorem from Tarski's  
Theorem of the Undefinability of Truth. Thus, we at least reduce the self-referential arguments needed for
the Second Incompleteness Theorem and Tarski's  
Theorem of the Undefinability of Truth taken together.

We take self-reference, in this paper, as broadly as possible. There are all kind of subtle discussions that we wish to side-step:  \emph{are
Yablo-style constructions self-reference? does the Grelling paradox involve true self-reference?}
etcetera. We want to avoid  even the slightest smell of diagonalization, self-reference, the recursion theorem, circularity or majorization. 

\subsection{Motivation}
Why is a proof of G\"odel's Second Incompleteness Theorem from Tarski's 
Theorem interesting? Our result is a step in a program to find self-reference-free proofs
both of the Second Incompleteness Theorem and of Tarski's  Theorem.
I do not think there are, at present, convincing self-reference-free proofs ---at least, not in the very broad sense of self-reference discussed above---
of either the {\gtwo} or of Tarski's Theorem. There is a self-reference-free proof, due to Kotlarski, of Tarski's
Theorem \emph{for the case of Uniform Biconditionals}. See \cite{kotl:inco04}. However, Kotlarski's attempted proof for non-uniform biconditionals
fails. Moreover, Kotlarski's proof uses majorization which could be viewed as a form of diagonalization. 

Why is the quest for a self-reference-free proof of these two great theorems important? Well, first there is the `because the
question is there' argument. It is always annoying when we have in essence just one kind of proof of a great theorem.
Secondly, a different proof may lead to new generalizations and new insights. For example, this research
could help us find a solution of Jan Kraj\'{\i}\v{c}ek's problem to find a proof of the non-interpretability of
the extension ${\sf PC}(A)$ of a consistent finitely axiomatized sequential theory $A$ with predicative comprehension in $A$ itself that
does not run via the \gtwo.

\subsection{About the Proof}
Our proof is roughly the result of combining a proof plan due to Kreisel (see \cite{smor:inco77}) with a proof plan due to 
Adamowic and Bigorajska (see \cite{adam:exis01}).

\medent
I give two variants of the proof. The first avoids the use of the third L\"ob condition to wit that provability implies provability of 
provability. The second does use the third L\"ob condition. 
Since the two variants share a lot of text, I will represent the first variant as the main line and I will give the additions for the second line in \textvtt{typewriter font}.

\begin{remark}
The ideal would be to study a new proof in a setting that is so general that we do not have the resources to even
prove things like the G\"odel Fixed Point Lemma. We do not aspire to this ideal here. The theories studied all have the
resources to do diagonalization.
\end{remark}

\subsection{Prerequisites}
The full proof presupposes 
some knowledge of relevant materials from the text books \cite{haje:meta91} and \cite{kaye:mode91}.
For the benefit of the reader who is not acquainted with weak theories we sketch the proof for the case
of Peano Arithmetic in Section~\ref{pa}. This special case has the advantage that almost all technical
complications disappear, while the essence of the proof-idea is still visible.

We will employ, for the full proof, the notations and conventions and elementary facts of \cite{viss:look17} to which the present paper is a sequel.
We will use the Interpretation Existence Lemma. For a careful exposition of this last result, see \cite{viss:inte17}. Finally, we will make use of
a result from \cite{viss:pean14}.

\section{The Second Incompleteness Theorem for Peano Arithmetic}\label{pa}

We prove {\gtwo} for Peano Arithmetic {\sf PA}. Let the standard axiomatization of {\sf PA} be $\pi$.
We write $\opr_\pi A$ for the arithmetization of the provability of $A$ from $\pi$ and
$\oco_\pi A := \neg\opr_\pi \neg A$ for the consistency of (the theory axiomatized by) $\pi$ extended with $A$.

\begin{theorem}\label{gpa}
${\sf PA} \nvdash \oco_\pi \top$.
\end{theorem}

\noindent
We will use Tarski's Theorem of the Undefinability of Truth for the case of {\sf PA}.

\begin{theorem}\label{tpa}
For any arithmetical formula $A(x)$ \textup{(}with only $x$ free\textup{)} the theory
${\sf PA}  + \verz{A(\gn B) \iff B \mid \text{$B$ is an arithmetical sentence}}$ is inconsistent.
\end{theorem}

\begin{proof}[Proof of Theorem~\ref{gpa} from Theorem~\ref{tpa}]
We assume Theorem~\ref{tpa}. Suppose, in order to arrive at a contradiction, that ${\sf PA} \vdash \oco_\pi \top$.

\medent
Let $\mathcal S$ be a maximal set of $\Sigma^0_1$-sentences such that
${\sf PA} + \mathcal S$ is consistent. Let $\mathcal M$ be a model of ${\sf PA}+\mathcal S$.

We apply the Henkin Construction as described in
e.g. \cite{fefe:arit60} (see also \cite{viss:inte17}), to construct an internal Henkin model ${\sf H}(\mathcal M)$
of {\sf PA} based on $\oco_\pi\top$. We claim that $\mathcal M$ and ${\sf H}(\mathcal M)$ satisfy the same
$\Sigma^0_1$-sentences, to wit the $\Sigma^0_1$-sentences of $\mathcal S$. 

Consider any $\Sigma^0_1$-sentence $S$ and suppose $\mathcal M \models S$. It is easily seen that
${\sf H}(\mathcal M)$ is an end-extension of $\mathcal M$ and, hence, ${\sf H}(\mathcal M) \models S$.
\seli{By inspection of the Henkin construction, we find that ${\sf PA} \vdash \opr_\pi B \to B^{\sf H}$. So, since,
by $\Sigma^0_1$-completeness, we have ${\sf PA}\vdash S \to \opr_\pi S$, we find  ${\sf PA} \vdash S \to S^{\sf H}$.
Thus, it follows that ${\sf H}(\mathcal M) \models S$.} 

We may conclude that ${\sf H}(\mathcal M) \models \mathcal S$.
On the other hand, by the maximality of $\mathcal S$, the models $\mathcal M$ and ${\sf H}(\mathcal M)$ cannot satisfy more
$\Sigma^0_1$-sentences than those in $\mathcal S$.

\medent
Since, ${\sf H}(\mathcal M)$  again will contain  $\oco_\pi\top$ we can repeat the
construction {\sf H} to obtain ${\sf H}{\sf H}(\mathcal M)$.

We claim that ${\sf H}(\mathcal M)$ and ${\sf H}{\sf H}(\mathcal M)$ are elementary equivalent. The reason is as follows.
The Henkin construction is based on yes-no-decisions that depend on the truth or falsity of $\Sigma^0_1$-questions. On the 
standard level it, thus, depends on the truth or falsity of $\Sigma^0_1$-sentences. Since $\mathcal M$ and
${\sf H}(\mathcal M)$ satisfy the same $\Sigma^0_1$-sentences we find that, on standard levels, the same choices are made
in the Henkin construction
and hence that  ${\sf H}(\mathcal M)$ and ${\sf H}{\sf H}(\mathcal M)$ are elementary equivalent. 

The Henkin construction provides an internal truth-predicate $\mathfrak H$ such that we have
${\sf PA} \vdash \mathfrak H(\gn B) \iff B^{\sf H}$. So we find:
\begin{eqnarray*}
{\sf H}(\mathcal M) \models B & \Iff & {\sf H}{\sf H}(\mathcal M)\models B \\
& \Iff & {\sf H}(\mathcal M) \models B^{\sf H} \\
& \Iff & {\sf H}(\mathcal M) \models \mathfrak H(\gn B).
\end{eqnarray*}
Thus, ${\sf H}(\mathcal M)$ is a model of ${\sf PA}+\verz{  \mathfrak H(\gn B) \iff B \mid \text{$B$ is an arithmetical sentence}}$.
But this contradicts Theorem~\ref{tpa} on the undefinability of truth.
\end{proof}

\noindent
We note that our construction effectively yields a $\Delta^0_2$-truth-predicate $\mathfrak H$ such that
${\sf PA}+ \verz{\mathfrak H(\gn{B}) \iff B \mid \text{$B$ is an arithmetical sentence}}$ is consistent, in case 
{\sf PA} proves its own consistency.

We also note that the proof works on the weaker assumption that ${\sf PA}+ \mathcal S \vdash \oco_\pi \top$. So, it
 follows that $\opr_\pi\bot$ is in any maximal set of $\Sigma^0_1$-sentences $\mathcal S$ such that
${\sf PA}+\mathcal S$ is consistent. 

\section{Statement of Two Theorems}

In this section, we state both  {\gtwo} and Tarski's Theorem of the Undefinability of Truth in the general forms
we  consider in the present paper.

\subsection{Our Version of the \gtwo}
We work with ${\sf S}^1_2$ as our basic basic weak arithmetic. See \cite{buss:boun86} or \cite{haje:meta91}.
We use $\opr_\sigma$ for the arithmetization of provability from axiom set $\sigma$ and
$\oco_\sigma$ for $\neg\opr_\sigma \neg$.

\medent
We will prove the following version of the \gtwo. 

\begin{theorem}\label{goedelsmurf}
Suppose $U$ is a recursively enumerable theory. Let $\sigma$ be a $\Sigma^0_1$-formula \seli{$\Sigma_1^{\sf b}$-formula}
 that defines an axiom set of $U$ in the standard model. Suppose that  $N:U \rhd ({\sf S}^1_2+\oco_\sigma \top)$.
Then $U$ is inconsistent.
\end{theorem}

\noindent
The basic idea of a version of the {\gtwo} using interpretations is due to Feferman \cite{fefe:arit60}.
Of course, Feferman, around 1960, was thinking of {\sf PA} as base theory rather than ${\sf S}^1_2$.

\medent
The variant where $\sigma$ is $\Sigma^0_1$ is stronger than the one where $\sigma$ is $\Sigma_1^{\sf b}$. However,
there are easy arguments to reduce the {\gtwo} for $\Sigma^0_1$-axiomatizations to the {\gtwo} for $\Sigma_1^{\sf b}$-axiomatizations.
See \cite{viss:look17}.

\subsection{Our Version of Tarski's Theorem}
We employ the following version of Tarski's Theorem.
Let $\Theta$ be a signature. Suppose $N: U \rhd {\sf R}$, where {\sf R} is the very weak arithmetic introduced in \cite{tars:unde53}.
We take ${\sf TB}^{N,A}_\Theta$ to be the set of all Tarski biconditionals $A(\gn B) \iff B$, for $B$ a $\Theta$-sentence
and $\gn{B}$ an $N$-numeral.

\begin{theorem}\label{tarskismurf}
Let $U$ be a theory of signature $\Theta$. Suppose $N: U \rhd {\sf R}$ and, for some $U$-formula $A$, we have
$U \vdash {\sf TB}^{N,A}_\Theta$. Then, $U$ is inconsistent.
\end{theorem}

\noindent
We can reduce our version of Tarski's Theorem to the following special case. This reduction uses the recursion theorem,
so since we want to exclude anything that even smells of self-reference we, probably,  have to exclude this reduction from our main line of argument.
This is a pity since it would have been nice to reduce all instances of the {\gtwo} to one single application of self-reference.

\begin{theorem}\label{tarskikleinsmurf}
Let $\mathcal A$ be the signature of arithmetic extended with a fresh unary predicate {\sf T}. 
Then, ${\sf R}+ {\sf TB}^{{\sf ID},{\sf T}}_{\mathcal A}$ is inconsistent.
\end{theorem}

\noindent
Here is our argument for the reduction of Theorem~\ref{tarskismurf} to Theorem~\ref{tarskikleinsmurf}.
Suppose $N:U \rhd {\sf R}$ and $U \vdash {\sf TB}^{N,A}_\Theta$. Let $\nu$ be the translation associated with $N$.
We interpret ${\sf R}+ {\sf TB}^{{\sf ID},{\sf T}}_{\mathcal A}$ in $U$ via the translation $\nu^\star$ which is $\nu$
 on the
arithmetical vocabulary and which is $A({\sf tr}_{\nu^\star}(v))$ on {\sf T}. Here, 
${\sf tr}_{\nu^\star}$ is the arithmetization of 
the function $B \mapsto B^{\nu^\star}$. We evidently need the recursion theorem to make our definition work.

\section{The Henkin Construction}
The main tool of our proof will be the Interpretation Existence Lemma. In this section we state this lemma and
collect the relevant facts around it.

\begin{theorem}
Suppose $\sigma$ is a $\Sigma^0_1$-formula that represents the axioms of $U$ in the standard model.
We have ${\sf H}[\sigma] : ({\sf S}^1_2 +\oco_\sigma \top)  \rhd U$.
  Here ${\sf H}[\sigma]$ is the Henkin interpretation that results from substituting $\sigma$ in certain  formulas  with unary second-order
  variable $X$ that together make up the Henkin interpretation ${\sf H}[X]$.
  
  We have, inside  ${\sf S}^1_2+\oco_\sigma\top$, a truth predicate $\mathfrak H[\sigma]$ for ${\sf H}[\sigma]$ that satisfies the commutation
  conditions for ${\sf H}[\sigma]$ on a definable cut ${\sf J}[\sigma]$.
  It follows that ${\sf S}^1_2 +\oco_\sigma \top \vdash \mathfrak H[\sigma](\gn A) \iff A^{{\sf H}[\sigma]}$.
\end{theorem}

\noindent
For a proof of this result see \cite{viss:inte17} which also discusses the history of the result.

\medent
Let $\Sigma^0_{1,1}$ be the class of formulas of the form $\exists x\, \forall y\, \leq tx\, \exists z\, S_0(x,y,z,\vec u)$,
where $S_0$ is $\Delta_0$. Inspection shows that ${\sf prov}_\sigma(x)$, where $\sigma$ is $\Sigma^0_1$,
 can be written as a $\Sigma^0_{1,1}$-formula.
\seli{Inspection shows that ${\sf prov}_\sigma(x)$, where $\sigma$ is $\Sigma^{\sf b}_1$,
 can be written as a $\exists\Sigma^{\sf b}_{1}$-formula. This employs the $\Sigma_1^{\sf b}$-collection (or: $\Sigma_1^{\sf b}$-replacement) Principle. 
 See \cite[Theorem 14, p53]{buss:boun86}.} We have the following insight.

\begin{theorem}\label{elementarysmurf}
Suppose $\sigma$ is a $\Sigma^0_1$-formula \seli{$\Sigma^{\sf b}_1$-formula} that represents the axioms of $U$ in the standard model.
Suppose $\mathcal N_0$ and $\mathcal N_1$ are models of
${\sf S}^1_2 +\oco_\sigma \top$. Suppose further that 
$\mathcal N_0$ and $\mathcal N_1$ are $\Sigma_{1,1}^0$-elementary  equivalent \seli{$\exists \Sigma_1^{\sf b}$-elementary equivalent}.
Then ${\sf H}[\sigma](\mathcal N_0)$ and ${\sf H}[\sigma](\mathcal N_1)$ are elementary 
equivalent.
\end{theorem}

\begin{proof}
Inspecting the Henkin construction, we see that it fully depends on $\Sigma_{1,1}^0$-decisions \seli{$\exists \Sigma_1^{\sf b}$-decisions}.
For standard formulas only the standard $\Sigma_{1,1}^0$-sentences  \seli{$\exists \Sigma_1^{\sf b}$-sentences} true or false in our models are relevant.
\end{proof}

\noindent
Consider any theory $U$ of signature $\Theta$. Suppose the axiom set of $U$ is represented by $\sigma$ in $\Sigma_1^{0}$ \seli{in $\exists\Sigma_1^{\sf b}$}.
Let $N:U \rhd ({\sf S}^1_2+\oco_\sigma\top)$. We write {\sf H} for ${\sf H}[\sigma]$ and $\mathfrak H$ for
$\mathfrak H[\sigma]$.

\begin{theorem}\label{grotesmurf}
Consider $\mathcal M \models U$ and suppose $N(\mathcal M)$ and $N{\sf H}N(\mathcal M)$ are
$\Sigma^0_{1,1}$-elementary equivalent \seli{$\exists\Sigma^{\sf b}_{1}$-elementary equivalent}. Then ${\sf H}N(\mathcal M) \models {\sf TB}_\Theta^{N,\mathfrak H^N}$.
\end{theorem}
\begin{proof}
From the assumptions of the theorem, we find, by Theorem~\ref{elementarysmurf}, 
that ${\sf H}N(\mathcal M)$ and ${\sf H}N{\sf H}N(\mathcal M)$ are elementary equivalent.
Hence:
\begin{align*}
{\sf H}N(\mathcal M) \models A & \Iff \; {\sf H}N{\sf H}N(\mathcal M) \models A \\
& \Iff\;  {\sf H}N(\mathcal M) \models A^{{\sf H}N} \\
& \Iff  \; {\sf H}N(\mathcal M) \models \mathfrak H^N(\gn A). \qedhere
\end{align*}
\end{proof}

\noindent
\seli{Here is one final fact about the Henkin construction that we will use for the second line in the main proof.

\begin{theorem}
Suppose $\sigma$ is a $\Sigma^0_1$-formula \seli{$\Sigma^{\sf b}_1$-formula} that represents the axioms of $U$ in the standard model.
Then we have ${\sf S}^1_2+\oco_\sigma\top \vdash \opr_\sigma A \to \mathfrak H[\sigma](A)$. 
\end{theorem}

\noindent
The proof is by inspection of the Henkin construction.}

\section{The Proof}
Theorem~\ref{grotesmurf} tells us that, in order to to prove Theorem~\ref{goedelsmurf},  it is sufficient
 to provide a model $\mathcal M$ of $U$ so that $N{\sf H}[\sigma]N(\mathcal M)$ is
$\Sigma_{1,1}^0$-elementary equivalent \seli{$\exists\Sigma_{1}^{\sf b}$-elementary equivalent} with $N(\mathcal M)$.

\medent
For the first line, we use the following fact.

\begin{theorem}
We have: \[ {\sf S}^1_2 \rhd_{{\sf loc},{\sf cut}} {\sf W} := {\sf S}^1_2 + \verz{S \to S^I \mid S\in \Sigma^0_{1,1} \text{ and $I$ is a definable cut}\,}.\]
\end{theorem}

\noindent
Here $ \rhd_{{\sf loc},{\sf cut}}$ is local interpretability on definable cuts.
For a treatment of this insight, see \cite{viss:pean14}.\footnote{The theory {\sf W} is interpretable in an $\omega_1$-cut of the theory
Peano Basso of  \cite{viss:pean14}.}

\bident
Suppose $\sigma$ is a $\Sigma^0_1$-formula that represents the axioms of $U$ in the standard model.
Suppose $N:U \rhd ({\sf S}^1_2+\oco_\sigma \top)$. Since $\Pi^0_{1,1}$-formulas are downwards preserved to cuts,
we have
$({\sf S}^1_2 + \oco_\sigma\top) \rhd_{{\sf loc},{\sf cut}} ({\sf W}+ \oco_\sigma \top)$. It follows that
$U \rhd_{\sf loc}  ({\sf W}+ \oco_\sigma \top)$. Since $U$ proves its own consistency with respect to $N$ it follows that
$U$ is reflexive. Hence, by the Orey-H\'ajek characterization, we find $K: U \rhd  ({\sf W}+ \oco_\sigma \top)$, for some
$K$. We note that the Orey-H\'ajek characterization is itself based on the Interpretation Existence Lemma.

We now take $\mathcal S$ a maximal set of $\Sigma^0_{1,1}$-sentences such that $U+\mathcal S^K$ is consistent.
Let $\mathcal M$ be a model of $U+\mathcal S^K$. Let $\mathcal J$ be a $K$-internally definable cut between
$K(\mathcal M)$ and $K{\sf H}K(\mathcal M)$, where ${\sf H} := {\sf H}[\sigma]$. 
We  have:
 \begin{eqnarray}
 S \in \mathcal S & \To & K(\mathcal M) \models S \\
 & \To & \mathcal J \models S \\ 
 & \To & K{\sf H}K(\mathcal M) \models S 
 \end{eqnarray}

\noindent
Step (2) holds since  $K(\mathcal M) \models W$ and the definition of $\mathcal J$ is $K$-internal. We have Step (3), since
$\Sigma^0_{1,1}$-sentences are upwards preserved from cuts. 

By the maximality of $\mathcal S$, it follows that $S\in \mathcal S$ iff 
$K(\mathcal M)\models S$ and $S\in \mathcal S$ iff 
$K{\sf H}K(\mathcal M) \models S$. 
We may conclude that $K(\mathcal M)$ and $K{\sf H}K(\mathcal M)$ are $\Sigma^0_{1,1}$-elementary
equivalent. Thus, we have our desired result with $K$ in the role of $N$.

\medent
\seli{We treat the second line. Suppose $\sigma$ is a $\Sigma^{\sf b}_1$-formula that represents the axioms of $U$ in the standard model.
Suppose $N:U \rhd ({\sf S}^1_2+\oco_\sigma \top)$. Let $\mathcal S$ be a maximal set of 
$\exists\Sigma^{\sf b}_1$-sentences such that $U+\mathcal S^N$ is consistent.
Let ${\sf H} := {\sf H}[\sigma]$.
We have, by verifiable $\exists\Sigma_1^{\sf b}$-completeness in ${\sf S}^1_2$,  for $S\in \mathcal S$, that:
\begin{eqnarray*}
U+\mathcal S^N & \vdash & \opr_\sigma S^N \\
& \vdash & S^{N{\sf H}}
\end{eqnarray*}
Now let $\mathcal M$ be a model of $U+\mathcal S^N$. It follows that ${\sf H} N (\mathcal M) \models \mathcal S^N$. 
By maximality, it follows that, for any $S\in \exists\Sigma_1^{\sf b}$, we have 
\[ N(\mathcal M) \models S \;\Iff\; s\in \mathcal S \;\;\text{ and } \;\;
N {\sf H}N(\mathcal M) \models S \;\Iff\; s\in \mathcal S.\] So we are done.}

\medent
We note that the second line looks somewhat more efficient. However, 
the first line avoids the more refined syntactic analysis that is the basis
of the second line. 

\begin{question}
Our argument is presented as a model-theoretic argument. So, it is itself not obviously formalizable in a weak theory.
However, it seems to me that the models can be eliminated from the argument. They mainly function as a heuristic tool.
So, the question is how much resources do we need to internalize our argument in a theory. Is ${\sf S}^1_2$ sufficient?

We note that, since L\"ob's Principle follows, in the classical case, from the Second Incompleteness Theorem and since
the Third L\"ob Condition follows from L\"ob's Principle, an internalized version of the first variant of our proof implies the
Third L\"ob Condition. 

\medent
A closely related question is whether our argument can be made constructive. This seems, at first sight, rather hopeless because of the
radically non-constructive character of the Henkin construction. However, one can reduce the question the Second Incompleteness Theorem
for constructive theories to the Second Incompleteness Theorem for classical theories. Now if we can make the argument completely theory-internal
we would be there. 
\end{question}

\end{document}